\documentclass{article}

\usepackage{amsmath} 
\usepackage{amssymb}  
\usepackage{mathtools}		
\usepackage{amsthm}
\usepackage{xcolor}
\usepackage{balance}
\usepackage{subcaption}
\usepackage[margin=1.0in]{geometry}
\usepackage{hyperref}
\usepackage[ruled]{algorithm2e}

\newcommand{\R}{\mathbb{R}} 
\newcommand{\K}{\mathbb{K}} 
 
\newcommand{\N}{\mathbb{N}}

\newcommand{\Lie}{\mathcal{L}}

\newcommand{\M}{\mathbb{M}}
\newcommand{\scL}{\mathscr{L}}

\DeclarePairedDelimiter{\abs}{\lvert}{\rvert}

\DeclarePairedDelimiter{\ceil}{\lceil}{\rceil}

\DeclarePairedDelimiterX{\inp}[2]{\langle}{\rangle}{#1, #2}

\DeclarePairedDelimiter{\Mp}{\mathcal{M}_+(}{)}

\newtheorem{thm}{Theorem}[section]

\newtheorem{rmk}{Remark} 
\usepackage{mathrsfs}

\usepackage{color}
\usepackage{ulem}

\title{\LARGE \bf
Peak Estimation for Uncertain and Switched Systems
}
\renewcommand\footnotemark{}

\author{Jared Miller$^1$, Didier Henrion $^2$, Mario Sznaier$^1$, Milan Korda$^2$
\thanks{$^1$J. Miller and M. Sznaier are with the Robust Systems Lab,  ECE Department, Northeastern University, Boston, MA 02115. (e-mails: miller.jare@northeastern.edu,  msznaier@coe.neu.edu).}
\thanks{$^2$D. Henrion and M. Korda are with LAAS-CNRS, Universit\'e de Toulouse, CNRS, Toulouse, France. They are also with the Faculty of Electrical Engineering of the Czech Technical University in Prague, Czechia. (e-mail: henrion@laas.fr, korda@laas.fr)}
\thanks{ J. Miller and M. Sznaier were partially supported by NSF grants  CNS--1646121, CMMI--1638234, ECCS--1808381 and CNS--2038493,  and AFOSR grant FA9550-19-1-0005. 
This material is based upon research supported by the Chateaubriand Fellowship of the Office for Science \& Technology of the Embassy of France in the United States.
}
\thanks{The work of D. Henrion and M. Korda was partly supported by the European Union's Horizon 2020 research and innovation programme under the Marie Sklodowska-Curie Actions, grant agreement 813211 (POEMA). The work of M. Korda was also partly supported by the Czech Science Foundation (GACR) under contract No. 20-11626Y and by the AI Interdisciplinary Institute (ANITI) funding, through the French "Investing for the Future PIA3" program under the Grant  agreement ANR-19-PI3A-0004.
}
}
\begin{document}

\maketitle
\pagestyle{plain}

\begin{abstract}
\label{sec:abstract}
Peak estimation bounds extreme values of a function of state along trajectories of a dynamical system. This paper focuses on extending peak estimation to continuous and discrete settings with  time-independent and time-dependent uncertainty. Techniques from optimal control are used to incorporate uncertainty into an existing occupation measure-based peak estimation framework, which includes special consideration for handling switching uncertainties. The resulting infinite-dimensional linear programs can be solved approximately with Linear Matrix Inequalities arising from the moment-SOS hierarchy.
\end{abstract}
\section{Introduction}
\label{sec:introduction}

Peak estimation under uncertainty aims to bound extreme values of a state function subject to an adversarial noise process. Examples include finding the maximum height of an aircraft subject to wind, the maximum voltage in a transmission line subject to thermal noise, and the maximum speed of a motor subject to impedance within a tolerance.
A system with finite-dimensional state $x \in \R^{N_x}$ evolves under Ordinary Differential Equation (ODE) dynamics defined by a locally Lipschitz vector field $f$ perturbed by uncertainty over the time-range $t \in [0, T]$. The time-independent uncertainty $\theta \in \Theta \subset \R^{N_\theta}$ is fixed (such as the unknown mass of a system component within tolerance), while the time-dependent uncertainty $w(t)$ may change arbitrarily in time within the region $W \subset \R^{N_w}$. Let $x(t \mid x_0, \theta, w(t))$ denote a trajectory in time starting from an initial point $x_0$ subject to uncertainties $(\theta, w(t))$. The uncertain peak estimation problem with variables $(t, x_0, \theta, w(t))$ may be posed as,
    \begin{align}
    P^* = & \max_{t \in [0, T],\, x_0 \in X_0,\,  \theta \in \Theta,\,  w(t)} p(x(t \mid x_0, \theta, w(t))) &\label{eq:peak_traj}\\
    & \dot{x}(t) = f(t, x(t), \theta, w(t)), \quad 
    w(t) \in W & \forall t \in [0, T]. \nonumber
    \end{align}
This paper produces an infinite-dimensional linear program (LP) in occupation measures to upper bound the quantity $P^*$ from \eqref{eq:peak_traj}. Measure-based peak estimation was introduced in \cite{helmes2001computing} and \cite{cho2002linear} for a stochastic setting, and was numerically approximated by a  discretizing set of finite LPs. The work in \cite{fantuzzi2020bounding} forms a sum-of-squares program from an LP dual to the measure LP in \cite{cho2002linear}.
Each of these are variations on the optimal control framework in \cite{lewis1980relaxation, lasserre2008nonlinear}, with an optimal stopping cost rather than an average integral (running) cost.

Occupation measure-based bounds for uncertain peak estimation may be developed by adapting methods from optimal control. 
Time-dependent uncertainty is an instance of an adversarial optimal control which aims to maximize the state function.
Time-independent parameter uncertainty may be incorporated by adding states, and switched systems can be analyzed by splitting the occupation measure \cite{henrion2012measures}.
The true peak cost $P^*$ is upper bounded with an infinite dimensional LP in occupation measures. The infinite LP is then truncated into a sequence of LMIs by the moment-SOS hierarchy \cite{lasserre2009moments}.

This paper has the following structure: Section \ref{sec:preliminaries} reviews preliminaries such as occupation measures and peak estimation. Section \ref{sec:uncertainty} presents uncertainty models, and a unified uncertain peak estimation model is presented in Section \ref{sec:unified}. Section \ref{sec:discrete} extends uncertain peak estimation to discrete systems. 
Section \ref{sec:safety} presents the application of uncertain peak estimation to safety analysis.
The paper is concluded in Section \ref{sec:conclusion}.
\section{Preliminaries}
\label{sec:preliminaries}

\subsection{Notation}
Let $\N$ be the set of natural numbers, $\R^n$ be an $n$-dimensional real Euclidean space, and $\R[x]$ be the set of polynomials in $x$ with real-valued coefficients. For a set $X \subseteq \R^n$, the sets $C(X)$ and $C_+(X)$ are respectively the set of continuous functions on $X$ and its nonnegative subcone. The subcone $C^1(X) \subset C(X)$ is composed of continuous functions on $X$ with continuous first derivatives. $\Mp{X}$ is the set of nonnegative Borel measures over $X$, and a duality pairing exists $\inp{f}{\mu} = \int_X f(x) d \mu(x)$ for all $f \in C(X), \ \mu \in \Mp{X}$. For every linear operator $\Lie$, there exists a unique linear adjoint $\Lie^\dagger$ such that $\inp{\Lie f}{\mu} = \inp{f}{\Lie^\dagger \mu}, \ \forall f,\mu$ is satisfied. An indicator function is $I_A(x)=1$ for a subset $A \subseteq X$ if $x \in A$ and $I_A(x) = 0$ otherwise. The measure of a set $A \subseteq X$ with respect to $\mu$ is $\mu(A) = \int_{A} d \mu =\int_X I_A(x) d\mu$. The quantity $\mu(X) = \inp{1}{\mu}$ is known as the `mass', and $\mu$ is a probability measure if $\mu(X) = 1$. The Dirac delta $\delta_{x'} \in \Mp{X}$ is a probability measure supported only on $x=x'$. For measures $\mu \in \Mp{X}, \ \nu \in \Mp{Y}$, the product measure satisfies $(\mu \otimes \nu)(A \times B) = \mu(A)\nu(B)$ for all $A \in X, \ B \in Y$. The projection map $\pi^x: X \times Y \rightarrow X$ returns only the $x$ coordinate $(x, y) \rightarrow x$. The pushforward by a function $f$ is the linear operator $f_\#$ satisfying $\inp{v(x)}{f_\# \mu} = \inp{v(f(x))}{\mu}$ for any test function $v \in C(X)$ and measure $\mu \in \Mp{X}$. The $x$-marginal of a measure $\mu \in \Mp{X \times Y}$ may be expressed as the pushforward of a projection $\pi_\#^x \mu$ with duality pairing $\inp{v(x)}{\pi_\#^x \mu} = \int_{X \times Y} v(x) d \mu(x, y) $ holding for all test functions $v(x) \in C(X)$.

\subsection{Peak Estimation and Occupation Measures}
\label{sec:peak}

The standard (no uncertainty) peak estimation setting involves a trajectory $x(t \mid x_0)$ starting at the initial point $x_0 \in X_0 \subset X$ evolving according to dynamics $\dot{x}(t) = f(t, x(t))$ in a space $X$. The program to find the maximum value of a state function $p(x)$ along trajectories is,
\begin{equation}
    \begin{aligned}
    P^* = & \max_{t \in [0, T],\, x_0 \in X_0} p(x(t \mid x_0)),  &\label{eq:peak_traj_std} \dot{x}(t) = f(t, x(t))
    \end{aligned}
\end{equation}



The extremum $P^*$ may be bounded through the use of occupation measure relaxations \cite{cho2002linear}. 
An optimal trajectory satisfying $P^* = p(x^*) = p(x(t^* \mid x_0^*))$ is described by a triple $( x_0^*, t^*, x^*)$ \cite{miller2020recovery}.
The initial probability measure $\mu_0 \in \Mp{X_0}$ is distributed over the set of initial conditions. The peak probability measure $\mu_p \in \Mp{[0, T] \times X}$ is a free-time terminal measure. 
For an optimal stopping time $t^*$ and subsets $A \subseteq [0, t^*], \ B \subseteq X$, the $\mu_0$-averaged occupation measure $\mu \in \Mp{[0, T] \times X}$ has a definition \cite{cho2002linear},
\begin{equation}
    \label{eq:avg_free_occ}
    \mu(A \times B) = \int_{[0, t^*] \times X_0} I_{A \times B}\left((t, x(t \mid x_0))\right) dt \, d\mu_0(x_0)
\end{equation}
The measure $\mu(A \times B)$ yields the average amount of time a trajectory with initial condition $x_0$ drawn from $\mu_0$ will spend in the region $A \times B$.

The Lie derivative operator $\Lie_f$ may be defined for all test functions $v \in C^1([0, T]\times X)$,
\begin{equation}
    \Lie_f v(t, x) = \partial_t v(t,x) + f(t,x) \cdot \nabla_x v(t,x)
\end{equation}
The three measures ($\mu_0, \ \mu_p, \ \mu$) are linked by Liouville's equation for all test functions, 
\begin{align}
\inp{v(t, x)}{\mu_p} &= \inp{v(0, x)}{\mu_0} + \inp{\Lie_f v(t,x)}{\mu} \label{eq:liou_int}
\end{align}

Liouville's equation ensures that initial conditions distributed as $\mu_0$ are connected to terminal points distributed as $\mu_p$ by trajectories following the polynomial vector field $f$.
Two consequences of \eqref{eq:liou_mom} are that $\inp{1}{\mu_0} = \inp{1}{\mu_p}$ ($v(t,x) = 1$) and that $\inp{1}{\mu} = \inp{t}{\mu_p}$ ($v(t, x) = t$).
Equation \eqref{eq:liou_int} may be expressed in a weak sense using the adjoint relaton $\inp{\Lie_f v}{\mu} = \inp{v}{\Lie^\dagger_f \mu}$,
\begin{align}
\mu_p &= \delta_0 \otimes \mu_0 + \Lie_f^\dagger \mu. \label{eq:liou_mom}
\end{align}

A convex measure relaxation of problem \eqref{eq:peak_traj_std} is,
\begin{subequations}
\label{eq:peak_meas}
\begin{align}
p^* = & \ \textrm{max} \quad \inp{p(x)}{\mu_p} \label{eq:peak_meas_obj} \\
    & \mu_p = \delta_0 \otimes\mu_0 + \Lie_f^\dagger \mu \label{eq:peak_meas_flow}\\
    & \inp{1}{\mu_0} = 1 \label{eq:peak_meas_prob}\\
    & \mu, \mu_p \in \Mp{[0, T] \times X} \label{eq:peak_meas_peak}\\
    & \mu_0 \in \Mp{X_0}. \label{eq:peak_meas_init}
\end{align}
\end{subequations}
Constraint \eqref{eq:peak_meas_prob} ensures that both $\mu_0$ and $\mu_p$ are probability measures. The objective \eqref{eq:peak_meas_obj} is the expectation of $p(x)$ with respect to the peak measure $\mu_p$. Program \eqref{eq:peak_meas} has a dual problem  over continuous functions, 
\begin{subequations}
\label{eq:peak_cont}
\begin{align}
    d^* = & \ \min_{\gamma \in \R} \quad \gamma & \\
    & {\gamma} \geq {v(0, x)}  &  & \forall x \in X_0 \label{eq:peak_cont_init}\\
    & \Lie_f v(t, x) \leq 0 & & \forall (t, x) \in [0, T] \times X \label{eq:peak_cont_f}\\
    & v(t, x) \geq p(x) & & \forall (t, x) \in [0, T] \times X \label{eq:peak_cont_p} \\
    &v \in C^1([0, T]\times X) \label{eq:peak_cont_v}&
\end{align}
\end{subequations}

The variable $v(t,x)$ is termed an auxiliary function in \cite{fantuzzi2020bounding}, and is an upper bound on the cost function $p(x)$ by \eqref{eq:peak_cont_p}. The graph $(t, x(t \mid x_0))$ is contained in the sublevel set $\{(t, x) \mid v(t, x) \leq \gamma\}$ for all $x_0 \in X_0$. Programs \eqref{eq:peak_meas} and \eqref{eq:peak_cont} satisfy strong duality ($p^* = d^*$) when the set $[0, T] \times X$ is compact (Theorem C.20 of \cite{lasserre2009moments}). The measure solution produces an upper bound $p^* \geq P^*$, and this bound is tight with $p^* = P^*$ when the set $[0, T] \times X$ is compact (Sec. 2.3 of \cite{fantuzzi2020bounding} and \cite{lewis1980relaxation}).


The work in \cite{cho2002linear} approximates  Problems \eqref{eq:peak_meas} and \eqref{eq:peak_cont} by a discretized linear program over a fine mesh. The method in \cite{fantuzzi2020bounding} bounds \eqref{eq:peak_cont} with a sum-of-squares (SOS) relaxation of polynomial nonnegativity constraints. The SOS relaxation produces a converging sequence of upper bounds to $p^*=d^*$ when $[0, T] \times X$ is compact. Optimal trajectories can be  localized by sublevel sets of  $v(t,x)$ and $\Lie_f v(t, x)$  following the method in \cite{fantuzzi2020bounding}.
\subsection{Moment-SOS Hierarchy}
\label{sec:moment_sos}
The $\alpha$-moment of a measure $\mu$ for a multi-index $\alpha \in \N^n$ is $y_\alpha = \inp{x^\alpha}{\mu}$. 
The moment sequence $y$ is the infinite collection of moments $\{y_\alpha\}_{\alpha \in \N^n}$ of the measure $\mu$.
There exists a linear (Riesz) functional $L_y$ converting a polynomial $p(x) \in \R[x]$ into a linear combination of moments in $y$:
\begin{equation}
    L_y(p) = L_y\left(\textstyle \sum_{\alpha \in \N} p_\alpha x^\alpha \right) = \textstyle\sum_{\alpha \in \N} p_\alpha y_\alpha
\end{equation}

The moment matrix $\M[y]$ is a square symmetric matrix of infinite size and is indexed by monomials $(\alpha, \beta)$ as $\M[y]_{\alpha,\beta} = y_{\alpha + \beta}$ \cite{lasserre2009moments}. If a polynomial $p=\sum_{\alpha}p_\alpha x^\alpha$ with coefficients $p_\alpha$ is treated as a vector $\mathbf{p}$, evaluation of $\inp{p(x)^2}{\mu}$ is equivalent to $\mathbf{p}^T \M[y] \mathbf{p}$ by the Riesz functional $L_y$. 
Nonnegativity of $\inp{p(x)^2}{\mu}$ for all $p(x) \in \R[x]$ requires that $\M[y]$ is Positive Semidefinite (PSD). 

A basic semialgebraic set $\K = \{x \mid g_i(x) \geq 0, \ i = 1, \ldots, N_c\}$ may be the support set for a measure $\mu \in \Mp{\K}$. Because $\mu$ is supported over the region $\{x \mid g_i(x) \geq 0\}$, the evaluation $\inp{p(x)^2 g_i(x)}{\mu}$ is nonnegative for all polynomials $p(x) \in \R[x]$. The PSD localizing matrix associated with $g_i(x) \in \R[x]$ and the moment sequence $y$ is,
\begin{equation}
    \M[g_i y]_{\alpha, \beta} = \textstyle\sum_{\gamma \in \N^n} g_{i\gamma} y_{\alpha+ \beta + \gamma}.
\end{equation}
A necessary condition for a moment sequence $y$ to correspond with moments of a representing measure on $\K$ is that $\M[y]$ and all $\M[g_i y]$ are PSD. This necessary condition is sufficient if $\K$ is Archimedean \cite{putinar1993compact}.
A degree-$d$ finite truncation of these matrices keeps moments up to order $2d$, which are located in the upper-left corners of the infinite dimensional matrices. 
The truncated moment matrix $\M_d[y]$ has size $\binom{n+d}{d}$ corresponding to the monomials of $x$ with degree $\leq d$, and the localizing matrix $\M_{d - \textrm{deg}(g_i)}[g_i y]$ has size $\binom{n+d - \textrm{deg}(g_i)}{d-\textrm{deg}(g_i)}$.
An infinite dimensional LP in measures may be posed with a polynomial objective $p(x)$ and $m$ polynomial constraint functions $a_j(x) \in \R[x], \ \forall j=1, \ldots, m$ with $b\in\R^{m}$ as,
\begin{subequations}
\label{eq:meas_program}
\begin{align}
    p^* =&  \max_{\mu \in \Mp{X}} \inp{p}{\mu}\\
    &\inp{a_j(x)}{\mu} = b_j  & & \forall j=1,\ldots,m.
\end{align}
\end{subequations}
The degree-$d$ finite truncation of \eqref{eq:meas_program} is an LMI with an $\binom{n+2d}{2d}$-dimensional vector of moments $y$ as a variable,
\begin{subequations}
\begin{align}
\label{eq:mom_program}
    p^*_d &=  \ \max_{y} \textstyle\sum_\alpha p_\alpha y_\alpha \\
    & \M_d(y) \succeq 0, \ \M_{d - d_i}(g_{i} y) \succeq 0  & & \forall i = 1, \ldots, N_c\\
    & \textstyle\sum_\alpha a_{j \alpha} y_\alpha\ = b_j  & & \forall j = 1, \ldots, m.
\end{align}
\end{subequations}
Increasing $d$ results in a decreasing sequence of upper bounds $p_d^* \geq p_{d+1}^* \geq \ldots \geq p^*$, which is convergent if $\K$ is Archimedean. The refinement of upper bounds to \eqref{eq:meas_program} by LMIs of increasing complexity is the moment-SOS hierarchy \cite{lasserre2009moments}. The moment-SOS relaxation to the peak estimation program \eqref{eq:peak_meas} is available in Equation (15) of \cite{miller2020recovery}, which is an LMI in moment sequences $(y_0, y_p, y)$ up to degree $2d$ of the measures $(\mu_0, \mu_p, \mu)$. These moment relaxations are dual to the SOS programs in \cite{fantuzzi2020bounding}. Near-optimal trajectories extremizing $p(x)$ may be recovered from LMI solutions if the moment matrices obey rank conditions \cite{miller2020recovery}.


\section{Uncertainty Models}
\label{sec:uncertainty}
This section summarizes techniques for incorporating uncertainty into occupation-measure based frameworks, and briefly notes their application to peak estimation. The methods mentioned here arose from optimal control and the approximation of reachability sets. The two basic types of uncertainty are time-independent ($\theta \in \Theta$) and time-dependent ($w \in W$). It is assumed that $\Theta$ and $W$ are compact basic semialgebraic sets, just like $X$ and $X_0$. 

\subsection{Time-Independent Uncertainty}
\label{sec:time_indep}
Time-independent uncertainty $\theta_\ell$ for $\ell = 1 \ldots N_\theta$ may take values in a set $\Theta \subseteq \R^{N_\theta}$, and typically arises in systems with parameter tolerances. 
The time-independent $\theta$ may start at any value in $\Theta \subset \R^{N_\theta}$ and is then constant along trajectories. By the methods in \cite{lasserre2008nonlinear, henrion2012measures}, the state space may be extended into $X \times \Theta$ by adding new states $\theta$ with constant dynamics $\dot{\theta}_\ell = \Lie_f \theta_\ell = 0$ for each $\ell = 1 \ldots N_\theta$. 


\subsection{Time-Dependent Uncertainty}
\label{sec:time_dep}
Systems with time-dependent uncertainty may have the noise process $w(t)$ change arbitrarily quickly in $W$ over time $t$. Such bounded time-varying noise may be found in driving or piloting tasks with changing winds. 
The disturbance $w(t)$ is a Borel measurable function of time rather than the It\^o-type stochastic process  considered in \cite{helmes2001exittime}.
For an input $w(t) \in W$ and a subset $D \subseteq W$, the disturbance-occupation measure $\mu^w(A \times B \times D)$ is,
\begin{align}
    \label{eq:occ_measure_d}
    \int_{[0, T] \times X_0} I_{A \times B \times D}((t, x(t), w(t)) \mid x_0) dt \, d \mu_0(x_0). 
\end{align}
The disturbance $w(t)$ may be relaxed into a distribution $\omega(w \mid x, t)$, which is known as a Young Measure \cite{young1942generalized, lewis1980relaxation}.
The disturbance-occupation measure $\mu^w$ can be disentangled into $d\mu^w(t, x, w) = dt \, d\xi(x \mid t) \, d\omega(w \mid x, t)$ for conditional distributions $\xi, \omega$.
Liouville's equation with a relaxed disturbance $\omega(w \mid x, t)$ influencing dynamics $f(t, x, w)$ for all $v(t, x) \in C^1([0, T]\times X)$ is,
\begin{subequations}
\begin{align}
\label{eq:liou_d}
\inp{v(t, x)}{\mu_p} &= \inp{v(0, x)}{\mu_0} + \inp{\Lie_f v(t, x)}{\mu^w}. \end{align}
Equivalent expressions are formed by rearranging operators,
\begin{align}
\inp{v}{\mu_p} &= \inp{v}{\delta_0 \otimes \mu_0} + \inp{\Lie_f v}{\mu^w} & & \forall v \\
\inp{v}{\mu_p} &= \inp{v}{\delta_0 \otimes \mu_0} + \inp{v}{\Lie_f^\dagger \mu^w}&  & \forall v \label{eq:weak_before_sum}\\
\inp{v}{\mu_p} &= \inp{v}{\delta_0 \otimes \mu_0 + \pi^{tx}_\#\Lie_f^\dagger \mu^w} & & \forall v.\label{eq:weak_after_sum}
\end{align}
\end{subequations}

The measures of the two summands on the right hand side of \eqref{eq:weak_before_sum} reside in different spaces, as $\delta_0 \otimes \mu_0 \in \Mp{[0, T] \times X}$ while $\Lie_f^\dagger \mu^w \in \Mp{[0, T] \times X \times W}$. The $(t,x)$-marginalization $\pi^{tx}_\# \Lie^\dagger_f \mu^w \in \Mp{[0, T] \times X}$ allows the measures to be added together inside the duality pairing in \eqref{eq:weak_after_sum}. The duality pairings $\inp{v(t,x)}{\Lie^\dagger_f \mu^w}$ and $\inp{v(t,x)}{\pi^{tx}_\#\Lie^\dagger_f \mu^w}$ are equal for all $v \in C^1([0, T] \times X)$
because $v(t,x)$ is not a function of $w$. The weak disturbed Liouville's Equation is derived from \eqref{eq:weak_after_sum}  by treating  $\forall v(t, x) \in C^1([0, T]\times X)$ as implicit,
\begin{equation}
\label{eq:liou_disturb_weak}
    \mu_p = \delta_0 \otimes \mu_0 + \pi^{tx}_\# \Lie^\dagger_f \mu^w.
\end{equation}

Time-varying disturbances may be incorporated into peak estimation by letting $\mu \in \Mp{[0, T] \times X \times W}$ be a disturbance-occupation measure of the form in \eqref{eq:occ_measure_d} obeying a disturbed Liouville equation \eqref{eq:liou_disturb_weak}. The support sets  of the measures $\mu_0 \in \Mp{X_0}, \ \mu_p \in \Mp{[0, T] \times X}$ are unchanged when time-dependent uncertainty is added.


\subsection{Switching Uncertainty}

\label{sec:switching}



An approach for analyzing switched systems with occupation measures is presented in \cite{henrion2012measures}.
Let $\{X^k\}_{k=1}^{ N_s}$ be a closed cover of $X$ with $N_s$ switching modes. The sets  $X^k$ are not necessarily disjoint, and together satisfy $\cup_k X^k = X$ (definition of closed cover).
Each region $X^k$ has dynamics $\dot{x} = f_k(t,x)$ for some locally Lipschitz vector field $f_k$.
The closed cover formalism generalizes partitions of $X$  (deterministic dynamics) and arbitrary switching where $X^k = X \ \forall k$ (polytopic uncertainty). Polytopic uncertainty is a model with dynamics $f(t, x, k) = \sum_k w_k f_k(t, x)$ where the disturbance $w_k \in \R_+^{N_s}$ satisfies $\sum_k w_k = 1$. Trajectories from a switching system are equipped with a function $S: [0, T] \rightarrow 1 \ldots, N_s$ yielding the resident subsystem at time $t^-$. Such a trajectory under switching may be written as $x(t \mid x_0, S(t))$.
The switched measure program introduces an occupation measure $\mu_k \in \Mp{[0, T] \times X^k}$ for each subsystem $f_k$, 
\begin{equation}
    \mu = \textstyle\sum_k \mu_k \qquad \Lie^\dagger \mu = \textstyle\sum_k \Lie^\dagger_k \mu_k.
\end{equation}
A valid auxiliary function $v(t,x)$ from \eqref{eq:peak_cont_f} must decrease along all subsystems \cite{parrilo2008approximation, vlassis2014polytopic}.
Problem \eqref{eq:peak_cont} may be modified for switching by enlarging Constraint \eqref{eq:peak_cont_f} to,
\begin{equation}
    \Lie_{f_k} v(t,x) \leq 0 \quad \forall (t,x) \in [0, T] \times X_k, \ k = 1 \ldots N_s.\end{equation}

\begin{rmk}
The closed cover switching formalism may be expanded into a system with general time-dependent uncertainty if desired. The switching basic semialgebraic sets may be described as $X^k = \{x \mid g_{ki}(x)\geq 0 \ i = 1, \ldots, N_c^k\}$ for $N_c^k$ polynomial constraints each. 
A linear expression of time-dependent uncertain dynamics is $\dot{x}(t) = \sum_{k=1}^{N_s} w_k(t) f_k(t, x (t))$ for processes $w(t) \in \R^{N_s}_+$ satisfying $\sum_k w_k(t) = 1$ for all $t \in [0, T]$. Additional constraints must be imposed to enforce that the process $w_k(t)$ is zero whenever $x(t) \not\in X^k$. These constraints may be realized as $\{w_k g_{k i}(x) \geq 0, \ \forall i=1, \ldots, N_c^k, \ \forall k = 1, \ldots N_s\}$.
\end{rmk}
\section{Continuous-Time Uncertain Peak Estimation}
\label{sec:unified}
This section combines the uncertainty formulations from section \ref{sec:uncertainty} to form a pair of primal-dual infinite-dimensional LPs. The variables $\theta \in \Theta, w \in W$ will respectively denote time-independent and  time-dependent uncertainties of sizes $N_\theta, N_w$. The dynamics $f$  have $N_s$ switching subsystems $f_k(t, x, \theta, w)$ which are valid in regions $X_k \subseteq X$.

\subsection{Continuous-Time Measure Program}
A combined uncertain peak estimation measure program is detailed in Program \eqref{eq:peak_meas_un} with indices $k=1, \ldots, N_s$ for the switching subsystems,
\begin{subequations}
\label{eq:peak_meas_un}
\begin{align}
p^* = & \ \textrm{max} \quad \inp{p(x)}{\mu_p} & \label{eq:peak_meas_un_obj} \\
    & \mu_p = \delta_0 \otimes\mu_0 + \textstyle\sum_{k} \pi^{tx\theta}_\#\Lie_{f_{k}}^\dagger \mu_{k} \label{eq:peak_meas_un_flow}\\
    & \mu_0(X_0) = 1 & \label{eq:peak_meas_un_prob}\\
    & \mu_{k} \in \Mp{[0, T] \times X \times \Theta \times W} &\forall k\label{eq:peak_meas_un_occ}\\
    & \mu_p \in \Mp{[0, T] \times X \times \Theta} & \\
    & \mu_0 \in \Mp{X_0 \times \Theta}.& \label{eq:peak_meas_un_init}
\end{align}
\end{subequations}
\begin{thm}
The solution $p^*$ to program \eqref{eq:peak_meas_un}  will yield an upper bound to $P^*$ in \eqref{eq:peak_traj}.  \label{thm:peak_meas}
\end{thm}
\begin{proof}
First assume $N_s = 1$ with $X^1=X$, so there is only one switching domain. An optimal achievement of \eqref{eq:peak_traj} reaching the peak value of $P^*$ may be characterized by the tuple $(x_0^*, t^*, x_p^*, \theta^*, w^*(t))$. The peak value $p(x_p^*) = P^*$ is achieved by following the trajectory $x(t \mid x_0^*, \theta^*, w^*(t))$ until time $t=t^*$. Measures $ (\mu_0, \mu_p, \mu)$ may be defined from this optimal tuple such that the measures satisfy constraints
\eqref{eq:peak_meas_un_flow}-\eqref{eq:peak_meas_un_init}. The initial measure and peak measure may be set to $\mu_0 = \delta_{x=x_0^*}$ and $\mu_p =  \delta_{t=t^* }\otimes \delta_{x=x_p^*} \otimes \delta_{\theta = \theta^*}$ based on the optimal tuple. The measure $\mu \in \Mp{[0, T] \times X \times \Theta \times W}$ may be defined as the unique occupation measure satisfying,
\begin{equation}
    \inp{\tilde{v}}{\mu} =  \int_{t=0}^{t^*} \tilde{v}(t, x(t \mid x_0^*, \theta^*, w^*(t)), \theta^*, w^*(t)) dt,
\end{equation}
for all test functions $\tilde{v} \in C([0, T ] \times X \times \Theta \times W)$.
The measures $(\mu_0, \mu_p, \mu)$ satisfy constraints \eqref{eq:peak_meas_un_flow}-\eqref{eq:peak_meas_un_init}, so $p^* \geq P^*$ when $N_s = 1$. 

Optimal trajectories arising from a  system with $N_s > 1$ may be described in a tuple as $(x_0^*, t^*, x_p^*, \theta^*, w^*(t), S^*(t))$, where $S^*(t)$ is the sequence of switches undergone between times $t \in [0, t^*]$. The measures $\mu_0$ and $\mu_p$ may remain the same as in the non-switched case. Switching occupation measures $\mu_k$ may be set to the unique occupation measure supported on the graph $(t, x(t \mid x_0^*, \theta^*, w^*(t)), \theta^*, w^*(t))$ between times $t\in[0, t^*]$ when $S(t) = k$. These occupation measures satisfy constraints \eqref{eq:peak_meas_un_flow} and \eqref{eq:peak_meas_un_occ}, proving that there exists a feasible solution to \eqref{eq:peak_meas_un_flow}-\eqref{eq:peak_meas_un_init} with objective $P^*$ for the case of switching.
\end{proof}



\subsection{Continuous-Time Function Program}
Dual variables $v(t,x,\theta) \in C^1([0,  T] \times X \times \Theta)$ and $\gamma \in \R$ can be defined to find the Lagrangian of \eqref{eq:peak_meas_un}. 
\begin{align}
    \scL &= \inp{p(x)}{\mu_p} + \inp{v(t,x,\theta)}{\delta_0 \otimes\mu_0 + \textstyle\sum_{k} \pi^{tx\theta}_\# \Lie_{f_{k}}^\dagger \mu_{k}} \nonumber\\
    &+ \inp{v(t,x,\theta)}{- \mu_p} + \gamma(1 - \inp{1}{\mu_0}). \nonumber
\end{align}
The resulting dual program in $(v, \gamma)$ is,
\begin{subequations}
\label{eq:peak_cont_un}
\begin{align}
    d^* = & \ \min_{\gamma \in \R} \quad \gamma \label{eq:peak_cont_un_obj}\\
    & \forall (x, \theta) \in X_0 \times  \Theta: \nonumber \\
    & \quad {\gamma} \geq {v(0, x, \theta)}  &   \label{eq:peak_cont_un_init}\\
    & \forall (t, x, \theta, w) \in [0, T] \times X_k \times  \Theta \times W: \quad \forall k \nonumber \\
    & \quad \Lie_{f_{k}} v(t, x, \theta)\leq 0 \label{eq:peak_cont_un_flow}\\
    & \forall (t, x, \theta) \in [0, T] \times X \times  \Theta: \nonumber \\
    & \quad v(t, x, \theta) \geq p(x) \label{eq:peak_cont_un_p}  \\
    &v(t, x, \theta) \in C^1([0, T]\times X \times \Theta) \label{eq:peak_cont_un_v}.
\end{align}
\end{subequations}
\begin{thm}
There is no duality gap between \eqref{eq:peak_meas_un} and \eqref{eq:peak_cont_un} when the set $[0, T] \times X \times \Theta \times W$ is compact.
\label{thm:peak_strong}
\end{thm}
\begin{proof}
Necessary and sufficient conditions for there to be no duality gap between measure and function programs are if all measures are bounded and if the affine map is closed in the weak-* topology (Theorem C.20 of \cite{lasserre2009moments}).
A measure is bounded if all of its finite-degree moments are bounded. Boundedness will hold if the mass of the measure is bounded and the support of the measure is compact. In \eqref{eq:peak_meas_un} $\mu_0$ and $\mu_p$ each have mass 1, and the mass of $\sum_k \mu_k \leq T$ by Liouville's equation. Compactness of $[0, T] \times X \times \Theta \times W$ therefore assures that all measures are bounded. The image of the affine map $(\mu_0, \mu_p, \mu_k) \rightarrow (\delta_0 \otimes \mu_0 + \sum_k \pi_\#^{tx\theta} \Lie_{fk}^\dagger \mu_k - \mu_p,  \mu_0)$  induced by constraints \eqref{eq:peak_meas_un_flow}-\eqref{eq:peak_meas_un_prob} is closed in the weak-* topology. Strong duality therefore holds by closure and boundedness of measures.
\end{proof}

The measure $\mu_0$ has $N_x + N_\theta$ variables, and $\mu_p$ has $1 + N_x + N_\theta$ variables. The $N_s$ occupation measures $\mu_k$ each have $1 + N_x +  N_\theta + N_w$ variables. If the switching structure was not taken into account by the methods of section  \ref{sec:switching}, there would be a single occupation measure $\mu$ with $1 + N_x + N_\theta + N_w +  N_s$ variables. The affine uncertainty structure breaks up the large $\mu$ (in terms of the number of variables) into $N_s$ smaller measures $(\mu_k)$.

\subsection{Continuous-Time LMI Relaxation}

\label{sec:lmi_cont}
The compact (Archimedean) basic semialgebraic sets in the uncertain peak estimation setting are
\begin{subequations}
\label{eq:peak_sets}
\begin{align}
    X &= \{x \mid g_i(x)\geq 0 & &\mid \ i = 1, \ldots, N_c\} \\
    X_0 &= \{x \mid  g_{0i}(x)\geq 0 & &\mid \ i = 1, \ldots, N_c^0\}\\
    X^k &= \{x \mid g_{ki}(x)\geq 0 & &\mid \ i = 1, \ldots, N_c^k\} \\
    \Theta &= \{\theta \mid g_{\theta i}(\theta)\geq 0 & &\mid \ i = 1, \ldots, N_c^\theta\}\\
    W &= \{w \mid g_{wi}(w)\geq 0 & &\mid \ i = 1, \ldots, N_c^w\}.
\end{align}
\end{subequations}

The degree of $g_i(x)$ is $d_i$, and other degrees $d_{0i}, \ d_{\theta i}, \ d_{wi}, \ d_{ki}$ are defined on corresponding polynomials. Monomials forming moments may be indexed as  $x^\alpha t^\beta \theta^\gamma w^\eta$ for multi-indices $\alpha \in \N^{N_x}, \ \beta \in \N, \ \gamma \in \N^{N_\theta}, \ \eta \in \N^{N_w}$. Define $y^0 = \{y^0_{\alpha \gamma} \}, \ y^p = \{y^p_{\alpha \beta \gamma}\}$ as the moment sequences for measures $\mu_0$ and $\mu_p$. The moment sequence for the occupation measure $\mu^k$ is $y^k = \{y^k_{\alpha \beta \gamma \eta}\}$ for each switching subsystem $k$. 
The Liouville equation \eqref{eq:peak_meas_un_flow} with test function $v(t, x, \theta) = x^\alpha t^\beta \theta^\gamma$ has the form,
\begin{align}
\label{eq:liou_lmi_unc}
    0&= \inp{x^\alpha t^\beta \theta^\gamma}{ {\delta_0 \otimes \mu_0}} - \inp{x^\alpha t^\beta \theta^\gamma}{\mu_p} \\
    &+ \textstyle \sum_k\inp{\Lie_{f_{k}(t, x, \theta, w)}( x^\alpha t^\beta \theta^\gamma)}{\mu_k}. \nonumber
\end{align}
Define the operator $\textrm{Liou}_{\alpha \beta\gamma }(y^0, y^p, y^k)$ as the linear relation between the moment sequences induced by \eqref{eq:liou_lmi_unc} assuming that each $f_k$ is a polynomial vector field. Given a degree $d$,  define the degrees $d_k'$ as $d + \ceil{\textrm{deg}(f_k)/2} - 1$ for each $k$.
The degree-$d$ LMI relaxation of the uncertain peak estimation problem in  \eqref{eq:peak_meas_un} resulting in an upper bound $p^*_d \geq P^*$ is, 
\begin{subequations}
\label{eq:peak_lmi_un}
\begin{align}
    p^*_d = & \textrm{max} \quad \textstyle\sum_{\alpha} p_\alpha y_{\alpha 0 0}^p \label{eq:peak_lmi_un_obj} \\
    & \quad \textrm{Liou}_{\alpha \beta \gamma}(y^0, y^p, y^k) = 0 \quad \textrm{by \eqref{eq:liou_lmi_unc}} & &\forall \abs{\alpha} + \abs{\beta} + \abs{\gamma} \leq 2d   \label{eq:peak_lmi_un_flow}\\
    & y^0_0 = 1 \\
    & \M_d(y^0), \M_d(y^p), \forall k: \M_{d'_k}(y^k) \succeq 0 \label{eq:peak_lmi_un_psd} \\
    &\M_{d-2}(t(T-t)y^p) \succeq 0 \label{eq:peak_lmi_un_time}\\
    &\forall k: \M_{d'_k-2}(t(T-t)y^k) \succeq  0 \label{eq:peak_lmi_un_time2}\\
    & \M_{d - d_{0i}}(g_{0i} y^0) \succeq 0  & & \forall i = 1, \ldots, N_c^0  \label{eq:peak_lmi_un_init}  \\
    & \M_{d - d_{\theta i}}(g_{\theta i} y^0), \ \M_{d - d_{\theta i}}(g_{\theta i} y^p) \succeq 0 & & \forall i = 1, \ldots, N_c^\theta\\ 
    &\forall k: \M_{d_k'- d_{\theta i}}(g_{\theta i} y^k) \succeq 0 & & \forall i = 1, \ldots, N_c^\theta \label{eq:peak_lmi_un_init_w} \\
    &\M_{d - d_i}(g_{i} y^p) \succeq 0 & & \forall i = 1, \ldots, N_c\\ 
    &\forall k:  \M_{d_k'- d_{ki}}(g_{ki} y^k) \succeq 0 & & \forall i = 1, \ldots, N_c^k\\ 
    &\forall k: \M_{d - d_{wi}}(g_{wi} y^k) \succeq 0 & & \forall i = 1, \ldots, N_c^w. \label{eq:peak_lmi_un_d}
\end{align}
\end{subequations}

Constraints \eqref{eq:peak_lmi_un_psd}- \eqref{eq:peak_lmi_un_d} are moment and localizing matrix PSD constraints ensuring that there exist representing measures to the moment sequences $(y^0, y^p, y^k)$ supported on the appropriate spaces. The sequence $\{p^*_d\}$ will converge to $p^*$ monotonically from above
as $d\rightarrow \infty$ if all sets in \eqref{eq:peak_sets} are Archimedean \cite{lasserre2009moments}.


\subsection{Continuous-Time Uncertain Examples}

Code is available at \url{github.com/jarmill/peak}, and is written in Matlab R2020a using
Gloptipoly3 \cite{henrion2003gloptipoly}, 
YALMIP \cite{Lofberg2004}, 
and Mosek 9.2 \cite{mosek92} 
to formulate and solve LMIs.
Demonstrations are available in the folder \texttt{peak/experiments\_uncertain} and are run here on an Intel i9 CPU at 2.30 GHz with 64.0 GB of RAM. 

Dynamics based on Example 1 of \cite{xue2019inner} (adding $w$) are,
\begin{equation}
\label{eq:ex_inner}
    \dot{x}(t) = \begin{bmatrix}-0.5x_1 - (0.5 + w(t)) x_2 + 0.5 \\ -0.5 x_2 + 1 + \theta. \end{bmatrix}
\end{equation}

Figure \ref{fig:inner} illustrates maximization of $p(x) = x_1$ starting in $X_0 = \{x \mid (x_1 + 1)^2 + (x_2 + 1)^2 \leq 0.25\}$ for time $t \in [0, 10]$. The admissible disturbances $w(t)$ are in $w = [-0.2, 0.2]$. Fig. \ref{fig:inner_d} has $\Theta = 0$ while Fig. \ref{fig:inner_d_w} has $\Theta = [-0.5, 0.5]$ for the time-independent uncertainty $\theta \in \Theta$. In each figure, the black circles are initial conditions from the boundary of $X_0$, the blue curves are sampled trajectories, and the red plane are level sets for upper bounds of $x_1$ along trajectories. At the order $r=4$ LMI relaxation, Fig. \ref{fig:inner_d} yields a bound of $P^* \leq 0.4925 $ while Fig. \ref{fig:inner_d_w} with $\theta$ results in $P^* \leq 0.7680$. The black surface containing all trajectories in Fig. \ref{fig:inner_d} is the level set $\{(t,x) \mid v(t, x) = 0.4925\}$. 

\begin{figure}[ht]
     \centering
     \begin{subfigure}[b]{0.48\linewidth}
         \centering
         \includegraphics[width=\linewidth]{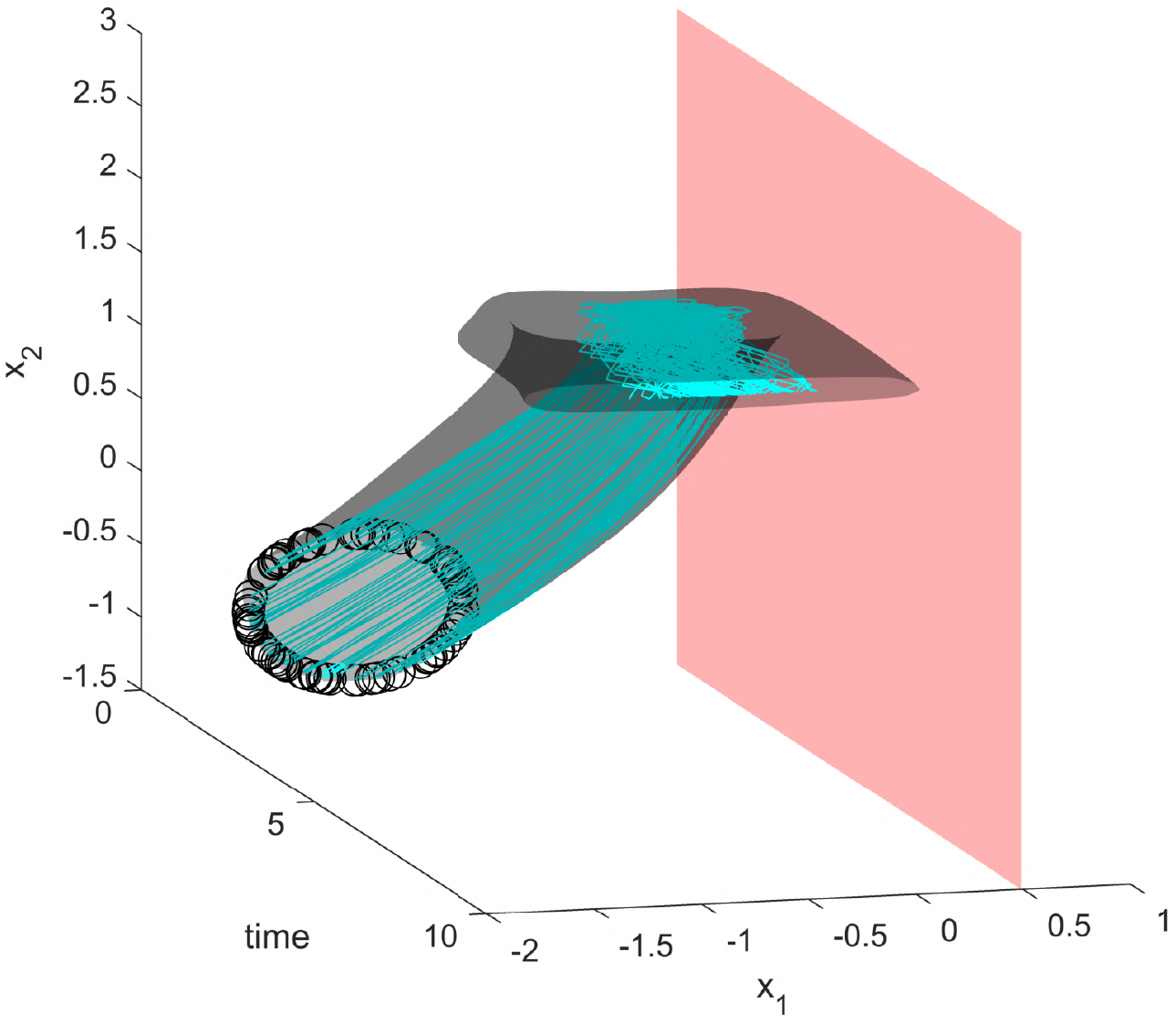}
         \caption{$\theta = 0$}
         \label{fig:inner_d}
     \end{subfigure}
     \;
     \begin{subfigure}[b]{0.48\linewidth}
         \centering
         \includegraphics[width=\linewidth]{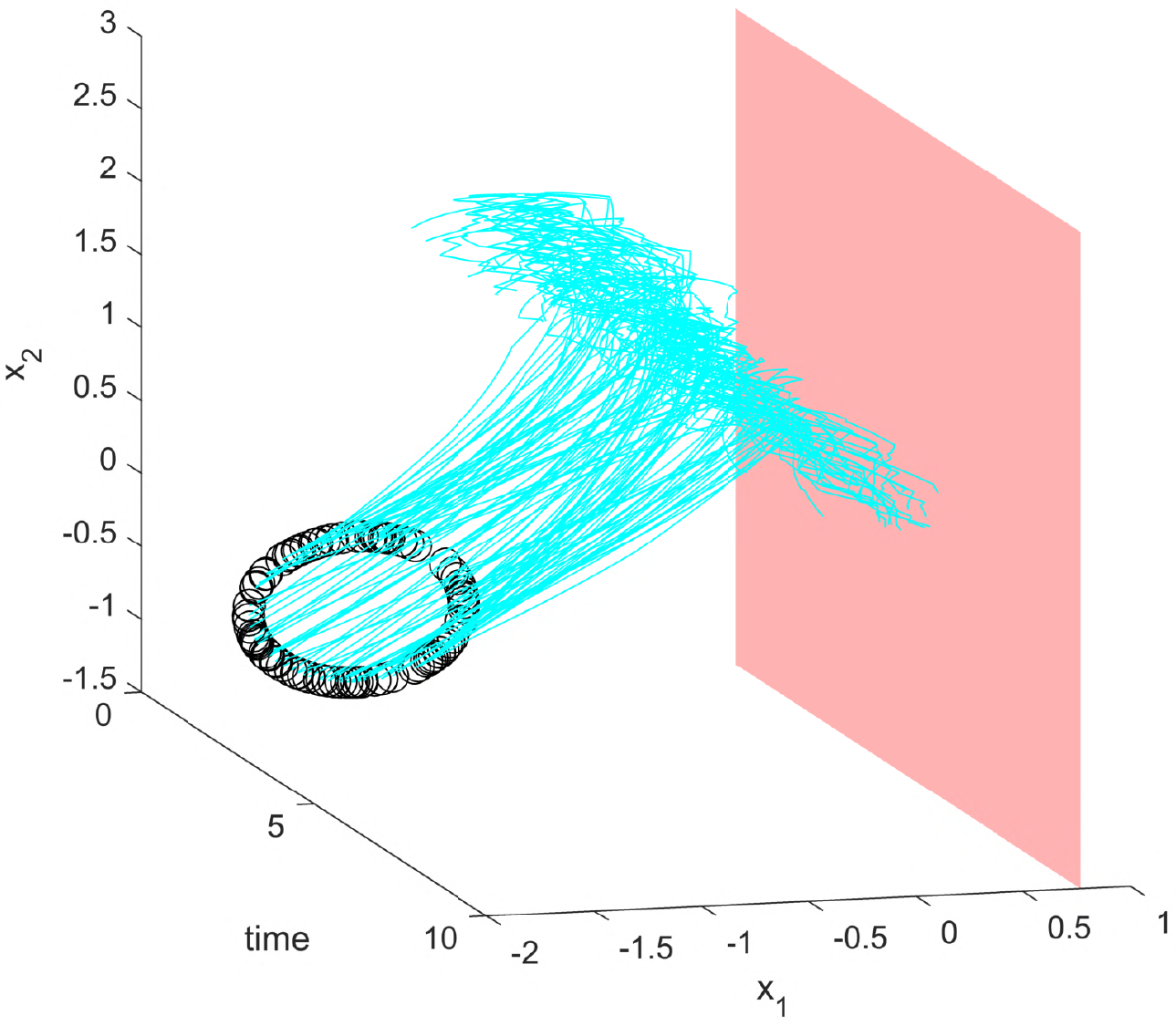}
         \caption{$\theta \in [-0.5, 0.5]$}
         \label{fig:inner_d_w}
     \end{subfigure}
      \caption{\label{fig:inner} Maximize $x_1$ at order 4 with $w(t) \in w$}
\end{figure}

The reduced three-wave model is a nonlinear model for the interaction of three quasisynchronous waves in a plasma \cite{goriely2001integrability}. These dynamics with parameters $(A, B, G)$ are,
\begin{align}
    \dot{x_1} &=  A x_1 + B x_2 + x_3 -2 x_2^2 \nonumber \\
    \dot{x_2} &= -B x_1 + A x_2 + 2 x_1 x_2 \label{eq:three_wave}\\
    \dot{x_3} &= -G x_3 - 2 x_1 x_2. \nonumber
\end{align}
This example aims to maximize $x_2$ on the three-wave system starting in $X_0 = \{x \mid (x_1+1)^2 + (x_2+1)^2 + (x_3+1)^2 \leq 0.16\}$. Order 3 LMI relaxations are used to upper bound $x_2$ over the region of interest $X = [-4, 3] \times [0.5, 3.6] \times [0, 4]$  and times $t \in [0, 5]$. The bound $P^* \leq 2.6108$ is produced with parameter values $A = 1, \; B = 0.5, \; G = 2$ (no uncertainty), as illustrated in Fig. \ref{fig:three_wave_std}. 
Fig. \ref{fig:three_wave_unc} adds uncertainty by letting $A \in [-0.5, 1.5]$ and $B \in [0.25, 0.75]$ vary arbitrarily with time, and $G$ now possesses parametric uncertainty in $[1.9, 2.1]$ . Uncertainty in $A, B$ are realized by switching between 4 subsystems of \eqref{eq:three_wave} with  $(A, B) \in \{0.5,1.5\} \times \{0.25, 0.75\}$
Uncertainty in $G$ is implemented as $G = 2 + \theta$ where $\theta \in [-0.1, 0.1]$. 
The order-3 bound under uncertainty in Fig. \ref{fig:three_wave_unc} is $P^* \leq 3.296$. 



\begin{figure}[ht]
     \centering
     \begin{subfigure}[b]{0.48\linewidth}
         \centering
         \includegraphics[width=\linewidth]{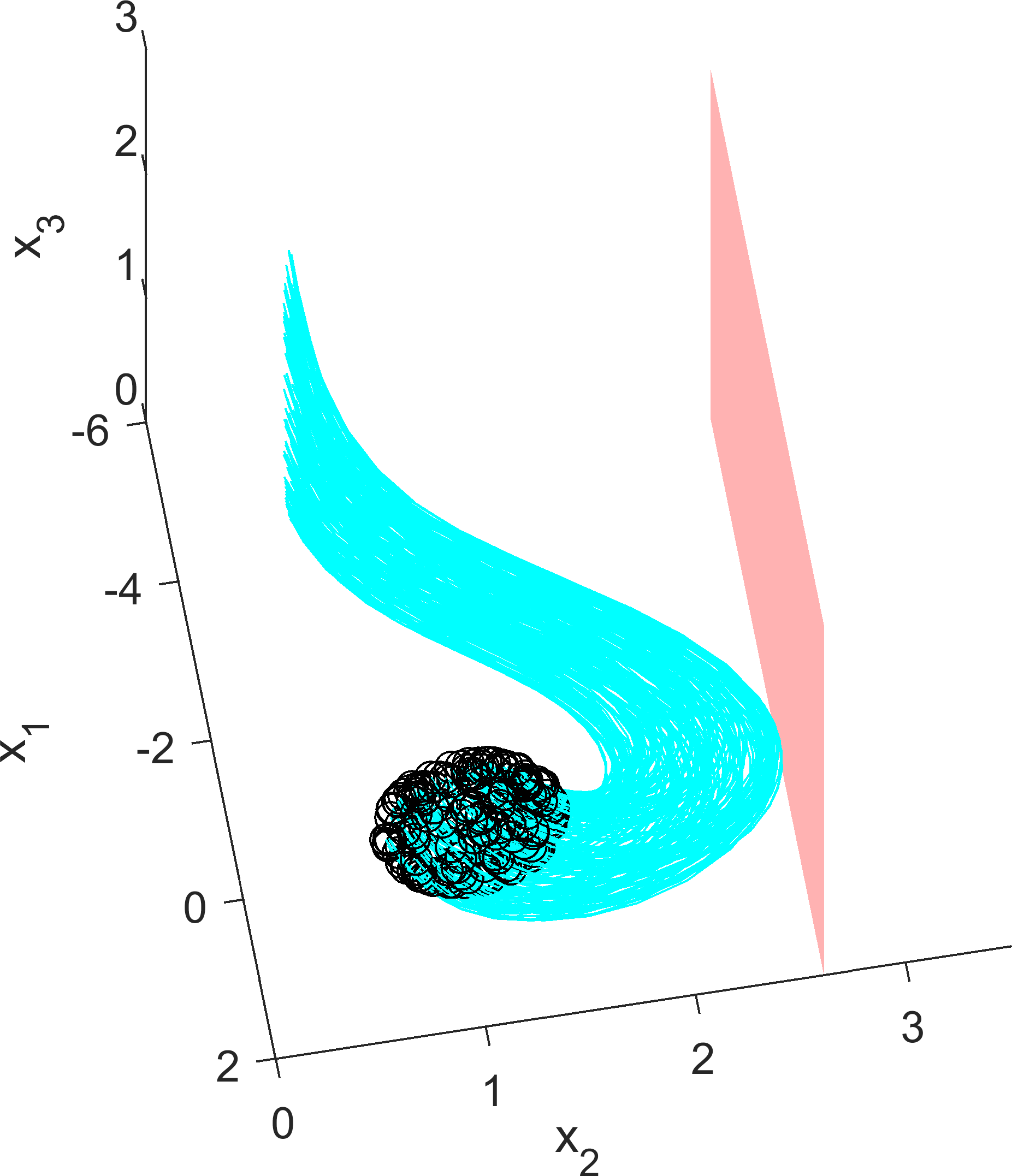}
         \caption{no uncertainty}
         \label{fig:three_wave_std}
     \end{subfigure}
     \;
     \begin{subfigure}[b]{0.48\linewidth}
         \centering
         \includegraphics[width=\linewidth]{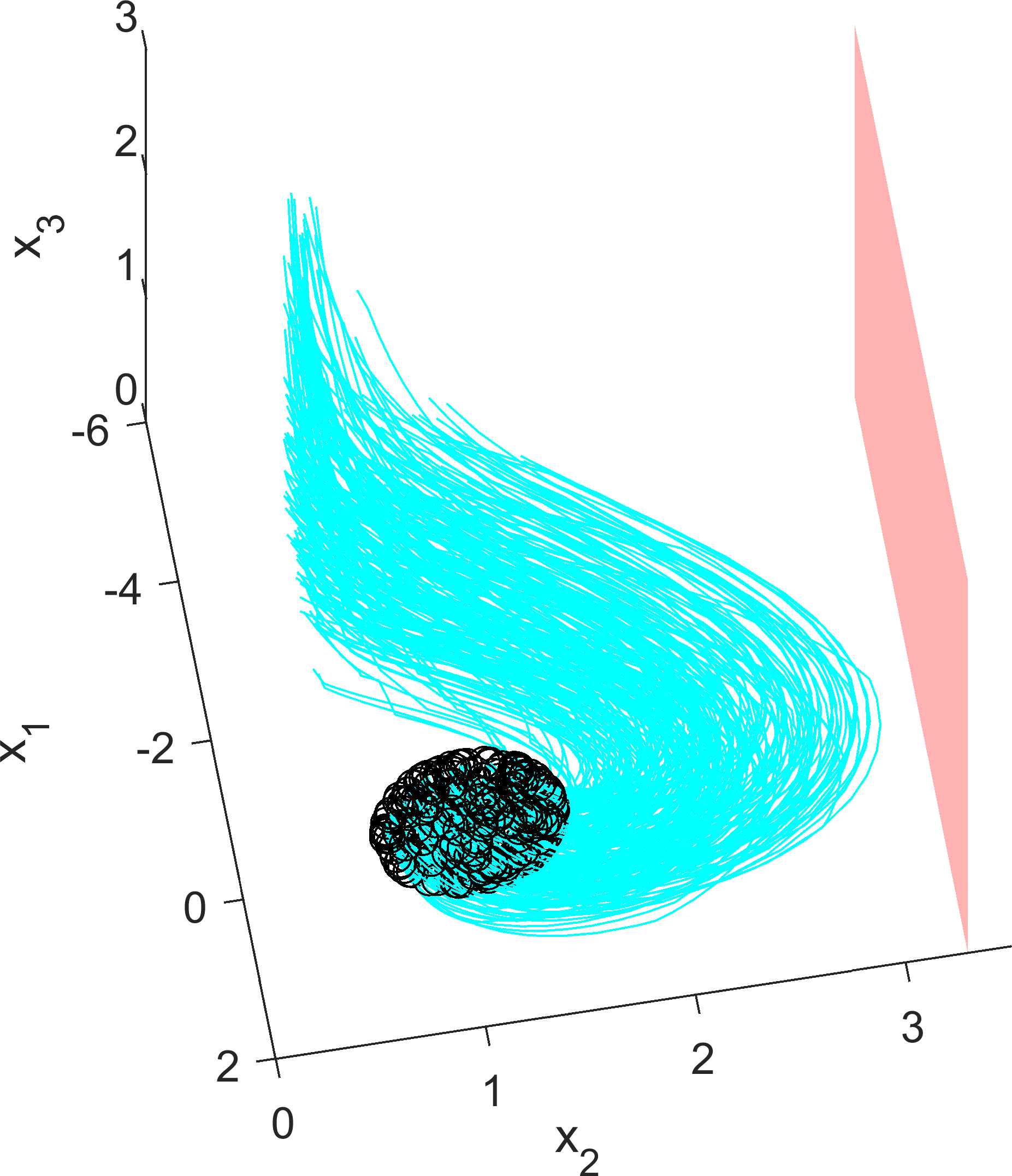}
         \caption{with uncertainty}
         \label{fig:three_wave_unc}
     \end{subfigure}
      \caption{\label{fig:three_wave} Maximize $x_2$ on three-wave system \eqref{eq:three_wave}
      }
\end{figure}





Section 4.1 of \cite{henrion2012measures} introduces a 1DOF attitude controller for validation of a space launcher system. These linearized dynamics corresponding to a double-integrator $I \ddot{\phi} = u$ and states $x = [\phi, \dot{\phi}]$. The input $u = \textrm{sat}_L(Kx)$ is a state feedback controller $Kx = 1000(2.475 \phi + 19.8 \dot{\phi})$ that saturates at levels $\pm L= \pm 380$. The subsystems are  linear operation $\abs{Kx} \leq L$, positive saturation $Kx \geq L$, and negative saturation $Kx \leq -L$ (deterministic switching).
These valid regions $X^k$ are separated in Fig. \ref{fig:space_all} by thin dotted diagonal lines. Maximizing $p(x) = \abs{\phi}$ (implemented as $\phi^2$) is shown in Fig. \ref{fig:space}. With $\abs{\phi_0} \leq 15^\circ$ and $\abs{\dot{\phi}_0} \leq 3^\circ/\textrm{sec}$, a degree-5 approximation finds a time-independent upper bound of $\abs{\phi_*} = 20.69^\circ$. The blue curve is the near-optimal trajectory, starting at the blue circle and extremizing $p(x)$ at the blue star.
The nominal moment of inertia in Fig. \ref{fig:space} is $I = 27,500 \ kg \;m^2$. Time-independent relative uncertainty $I$ may be introduced by replacing $I$ with $I/(1+\theta)$, where $\theta \in [-0.5, 0.5]$. The peak angle is raised to $\abs{\phi_*} = 51.86^\circ$ at $d=5$ at $I' = 2 I$ with this new uncertainty in Fig. \ref{fig:space_w}.

\begin{figure}[ht]
     \centering
     \begin{subfigure}[b]{0.48\linewidth}
         \centering
         \includegraphics[width=\linewidth]{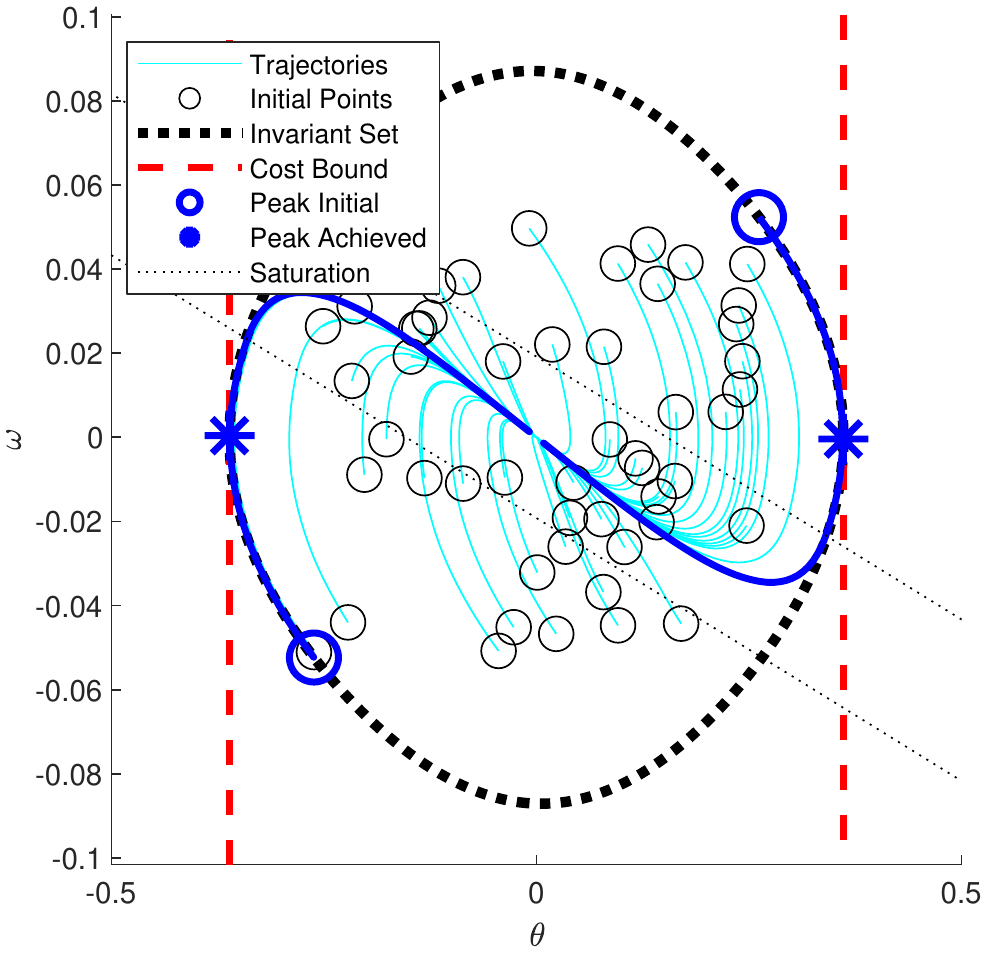}
         \caption{Certainty in $I$}
         \label{fig:space}
     \end{subfigure}
     \;
     \begin{subfigure}[b]{0.48\linewidth}
         \centering
         \includegraphics[width=\linewidth]{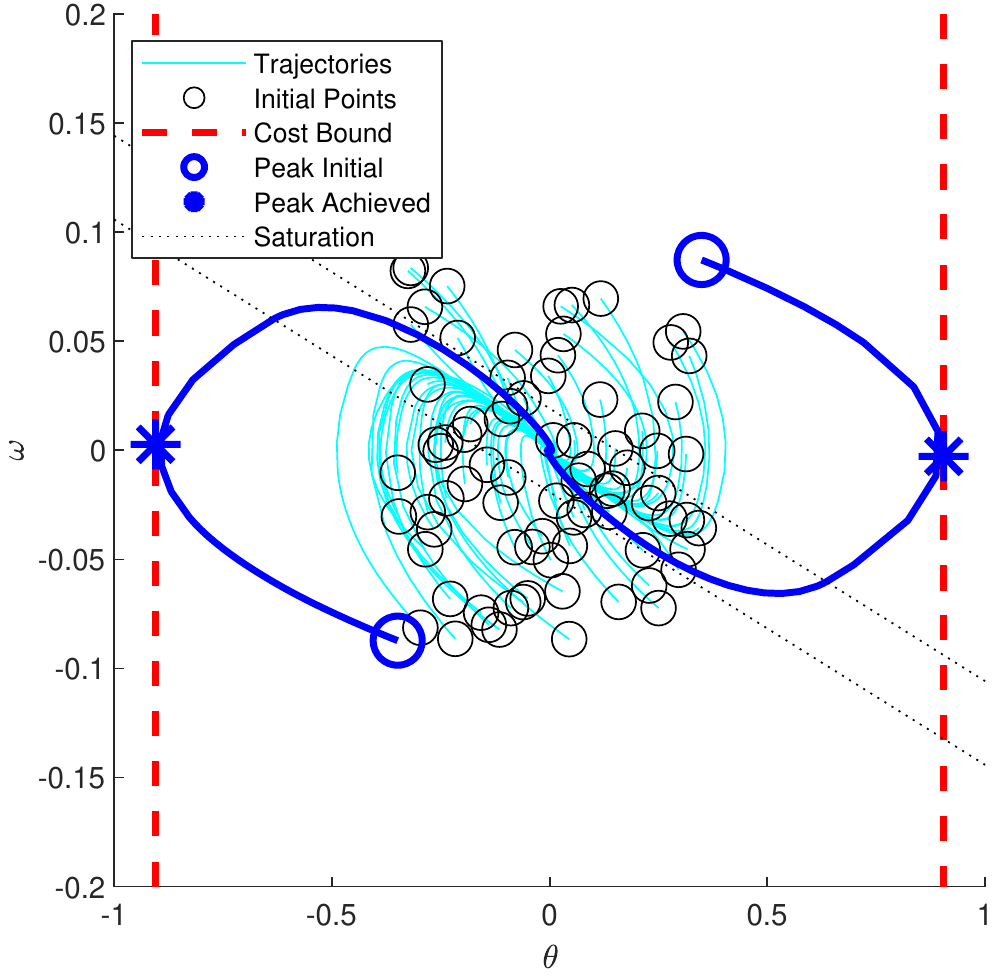}
         \caption{$\theta$ uncertainty in $I$}
         \label{fig:space_w}
     \end{subfigure}
      \caption{\label{fig:space_all} Maximum angle for 1DOF attitude controller \\}
\end{figure}
\section{Discrete-Time Uncertain Peak Estimation}
\label{sec:discrete}
Uncertain peak estimation can be extended to discrete systems, including switched discrete-time systems. 
A discrete system from times $t = 0, 1, \ldots, T$ is considered for  dynamics $x_+ = f(x)$ where $x_+$ is the next state. A trajectory starting at the initial condition $x_0 \in X_0$ is $x_t(x_0)$ The uncertain peak estimation problem for discrete systems with uncertainties $(\theta, w_t)$ and  $N_s$ subsystems with switching sequence $S_t$ is,
    \begin{align}
    P^* = & \max_{t, \, x_0 \in X_0,\,  \theta \in \Theta,\,  w_t, \ S_t} p(x_t(x_0, \theta, w_t, S_t)) &\label{eq:peak_traj_disc}\\
    & x_+ = f_k(x_t, \theta, w_t) \textrm{ if } S_t =  k \nonumber\\
    & w_t \in W, \ S_t \in 1, \ldots, N_s  \qquad \forall t \in 0, \ldots, T. \nonumber
    \end{align}
    
\subsection{Discrete-Time Measure Background}

Just as the Lie derivative $\Lie_f v$ yields the infinitesimal change in $v$ along continuous trajectories, the quantity $v(f(x)) - v(x)$ is the change in $v$ along a single discrete step. 
An occupation measure for sets $A \subseteq X$ with initial conditions distributed as $\mu_0 \in \Mp{X_0}$ may be defined for discrete systems,
\begin{equation}
    \label{eq:occ_discrete}
    \mu(A) = \int_{X_0} \sum_{t = 0}^T I_A(f^t (x_0)) d \mu_0.
\end{equation}
The quantity $\mu(A)$ is the averaged number of time steps that trajectories distributed as $\mu_0$ spend in the region $A$. For measures $\mu_0 \in \Mp{X_0}, \ \mu_p \in \Mp{X}, \ \mu \in \Mp{X}$, the strong and weak discrete Liouville equations for all $v$ are:
\begin{align}
    \inp{v(x)}{\mu_p} &= \inp{v(x)}{\mu_0} + \inp{v(f(x))}{\mu} - \inp{v(x)}{\mu}, \label{eq:liou_discrete_v} \\
    \mu_p &= \mu_0 + f_\# \mu - \mu . \label{eq:liou_discrete}
\end{align}
Time may be optionally included in system dynamics by setting a state $t_+ = t + 1$ and incorporating $t$ into dynamics. The pushforward term in \eqref{eq:liou_discrete} would then be $v(t+1, f(t,x)) - v(t,x)$.
Discrete systems with uncertainties $(\theta, w)$ have dynamics and Liouville equations according to,
\begin{align}
    x_{+} &= f(x_t, \theta, w_t),  & 
    \mu_p &= \mu_0 + \pi^{x \theta}_\#(f_\# \mu - \mu). \label{eq:liou_discrete_un}
\end{align}
The uncertainty $\theta \in \Theta$ is fixed, and the time-dependent uncertainty has $w_t \in W$ for every time step $t = 0, \ldots, T$. Switching uncertainty from Section \ref{sec:switching} with subsystems $f_k$ valid  over $X_k$ may be realized by defining occupation measures $\mu_k \in \Mp{X_k \times \Theta \times W}$ such that $\mu = \sum_k \mu_k$. 


\subsection{Discrete-Time Measure Program}
A measure program may be formulated to upper bound the peak-estimation task on discrete systems. The uncertainties available in this formulation are $(\theta, w)$ and switching between dynamics $f_k$ over $X_k$. 
The uncertain discrete peak estimation measure problem with variables $(\mu_0, 
\mu_k, \mu_p)$ is,
\begin{subequations}
\label{eq:peak_meas_un_disc}
\begin{align}
p^* = & \ \textrm{max} \quad \inp{p(x)}{\mu_p} \label{eq:peak_meas_un_disc_obj}&  \\
    & \mu_p = \mu_0 + \pi^{x \theta}_\#\left(\textstyle\sum_{k}( f_{k\#} \mu_{k} - \mu_k)\right) \label{eq:peak_meas_un_disc_flow}&    \\
    & \mu_0(X_0) = 1 & \label{eq:peak_meas_un_disc_prob}\\
    & T \geq \textstyle \sum_k \inp{1}{\mu_k} \label{eq:peak_meas_un_disc_time}\\
    & \mu_{k} \in \Mp{X_k \times \Theta \times W} & \forall k =1 , \ldots, N_s\\
    & \mu_p \in \Mp{X \times \Theta} & \\
    & \mu_0 \in \Mp{X_0 \times \Theta}.& \label{eq:peak_meas_un_disc_init}
\end{align}
\end{subequations}
\begin{rmk}
The composition of pushforwards in \eqref{eq:peak_meas_un_disc_flow} acts as $\inp{v(x, \theta)}{\pi^{x \theta}_\# f_{k \#} \mu_k} = \inp{v(f_k(x, \theta, w), \theta)}{\mu_k}$ for all test functions $v(x, \theta) \in C(X \times \Theta)$.
\end{rmk}
\begin{thm}
The optimum $p^*$ of \eqref{eq:peak_meas_un_disc} is an upper bound for $P^*$ from discrete program \eqref{eq:peak_traj_disc}.
\label{thm:peak_meas_disc}
\end{thm}
\begin{proof}
This proof follows the same steps as the proof to theorem \ref{thm:peak_meas}. An trajectory achieving a peak value of $P^*$ solving \eqref{eq:peak_traj_disc} may be expressed as a tuple $(t^*, x_0^*, x_p^*, \theta^*, w_t^*, S_t^*)$ with $P^* = p(x^*_p) =  p(x_{t^*}(x_0^*, \theta^*, w_t ))$. Measures may be defined from this tuple to solve problem \eqref{eq:peak_meas_un_disc}. The probability distributions are $\mu_0 = \delta_{x=x_0^*}$ and $\mu_p = \delta_{x =  x_p^* } \otimes \delta_{\theta = \theta^*}$. Switching  measures $\mu_k$ may be chosen as the unique occupation measures satisfying,
\begin{equation}
    \inp{\tilde{v}_k}{\mu_k} = \sum_{t=0}^{t^*}\tilde{v}(x_{t}(x_0^*, \theta^*, w_t^*), \theta^*, w_t^*) I(S_t = k), \label{eq:proof_occ_disc}
\end{equation}
for all test functions $\tilde{v}_k \in C(X_k \times \Theta \times W)$ and for each $k = 1, \ldots, N_s$. The measures $(\mu_0, \mu_p, \mu_k)$ are feasible solutions to \eqref{eq:peak_meas_un_disc_flow}-\eqref{eq:peak_meas_un_disc_init} with objective value  $P^* = p(x_p^*) = \inp{p(x)}{\mu_p}$, so $p^* \geq P^*$ is a valid upper bound to \eqref{eq:peak_traj_disc}.
\end{proof}
\begin{rmk}
Constraint \eqref{eq:peak_meas_un_disc_time} is a technique from \cite{magron2017discrete} ensuring that the maximal time in optimization is $T$ and that each $\mu_k$ has a bounded mass.
\end{rmk}

\subsection{Discrete-Time Function Program}
With dual variables $(v(x, \theta) \in C(X \times \Theta), \gamma \in \R)$ and a new dual variable $\alpha \geq 0$, the Lagrangian of \eqref{eq:peak_meas_un_disc} is,
\begin{align}
    \scL &= \inp{p(x)}{\mu_p} + \inp{v(x,\theta)}{\mu_0 - \mu_p}+ \alpha(T - \inp{1}{\textstyle \sum_k \mu_k}) \nonumber \\
    &+ \inp{v(x,\theta)}{\pi^{x\theta}_\# \textstyle\sum_{k} f_{k\#} \mu_k- \mu_{k}} + \gamma(1 - \inp{1}{\mu_0}).\nonumber
\end{align}
The corresponding dual problem is,
\begin{subequations}
\label{eq:peak_cont_un_disc}
\begin{align}
    d^* = & \ \min_{\gamma \in \R, \ \alpha \geq 0} \quad \gamma + T \alpha \label{eq:peak_cont_un_disc_obj}\\
    & \forall (x, \theta) \in X_0 \times \Theta: \nonumber \\
    & \quad {\gamma} \geq {v(x, \theta)}  &   \label{eq:peak_cont_un_disc_init}\\
    & \forall (x, \theta, w) \in  X_k \times \Theta \times W: \quad \forall k \nonumber \\
    & \quad v(f_k(x, \theta, w), \theta) -v(x, \theta)  \leq  \alpha \label{eq:peak_cont_disc_flow}\\
    & \forall (x, \theta) \in X \times \Theta: \nonumber \\
    & \quad v(x, \theta) \geq  p(x) \label{eq:peak_cont_disc_p}  \\
    &v(x, \theta) \in C(X \times \Theta) \label{eq:peak_cont_un_disc_v}.
\end{align}
\end{subequations}
\begin{thm}
Strong duality $p^*=d^*$ between holds between \eqref{eq:peak_meas_un_disc} and \eqref{eq:peak_cont_un_disc} if $T<\infty$ and $X \times \Theta \times W$ is compact. 
\label{thm:peak_strong_disc} 
\end{thm}
\begin{proof}
This is affirmed by a similar process to Theorem \ref{thm:peak_strong}. All measures have bounded finite moments given that their masses are bounded and their supports are compact. The image of the affine map in constraints \eqref{eq:peak_meas_un_disc_flow}-\eqref{eq:peak_meas_un_disc_prob} is closed in the weak-* topology,  concluding the conditions for strong duality by Theorem C.20 of \cite{lasserre2009moments}.
\end{proof}  

%

\subsection{Discrete LMI}

The LMI relaxation of \eqref{eq:peak_meas_un_disc} can be developed in the same manner as in Section \ref{sec:lmi_cont}. The sets $(X_0, X, W, D)$ are defined in the same way as in equation \eqref{eq:peak_sets}. As there is no $t$ term in discrete systems, monomials forming moments are indexed as $x^\alpha \theta^\gamma w^\eta$. The moment sequences are $y^0$, $y^p$, and a $y^k$ for each switching subsystem $k = 1, \ldots, N_s$. 
The Liouville equation \eqref{eq:peak_meas_un_disc_flow} with a given test function $v(x, \theta) = x^\alpha \theta^\gamma$  is ,
 \begin{align}
 \label{eq:liou_disc_lmi_unc}
    0&= \inp{x^\alpha \theta^\gamma}{ \delta_0 \otimes \mu_0} - \inp{x^\alpha  \theta^\gamma}{\mu_p} \\
    &+ \textstyle \sum_k\inp{( f_{k}(x, \theta, w)^\alpha \theta^\gamma - x^\alpha \theta^\gamma}{\mu_k}. \nonumber
\end{align}

The operator $\textrm{Liou}_{\alpha \gamma }(y^0, y^p, y^k)$  is defined as the relation induced by the discrete Liouville equation \eqref{eq:liou_lmi_unc_disc}. The discrete degree-$d$ LMI truncation of \eqref{eq:peak_meas_un_disc} is,
\begin{subequations}
\label{eq:peak_lmi_un_disc}
\begin{align}
    p^*_d = & \textrm{max} \quad \textstyle\sum_{\alpha} p_\alpha y_{\alpha 0}^p \label{eq:peak_lmi_un_disc_obj} \\
    & \textrm{Liou}_{\alpha \gamma}(y^0, y^p, y^k) = 0 \qquad \textrm{by \eqref{eq:liou_disc_lmi_unc}} & &  \forall \abs{\alpha} + \abs{\gamma} \leq 2d \label{eq:peak_lmi_un_disc_flow}\\
    & y^0_0 = 1 \\
    & \textstyle\sum_{k} y^k_0 \leq T \label{eq:peak_lmi_un_time_limit} \\
    & \M_d(y^0), \M_d(y^p), \forall k: \M_d(y^k) \succeq 0 \label{eq:peak_lmi_un_disc_psd} \\
    &\M_{d - d_{0i}}(g_{0i} y^0) \succeq 0  \label{eq:peak_lmi_un_disc_init}  & &\forall i = 1, \ldots, N_c^0\\ 
    &\M_{d - d_{\theta i}}(g_{\theta i} y^0) \succeq 0  \label{eq:peak_lmi_un_disc_init_w}  & &\forall i = 1, \ldots, N_c^{\theta}\\ 
    &\M_{d - d_i}(g_{i} y^p), \ \forall k: \M_{d - d_i}(g_{i} y^k) \succeq 0 & & \forall i = 1, \ldots, N_c \label{eq:peak_lmi_un_disc_peak} \\
    &\M_{d - d_{\theta i}}(g_{\theta i} y^p), \ \forall k: \M_{d - d_{\theta i}}(g_{\theta i} y^k) \succeq 0 & & \forall i = 1, \ldots, N_c^{\theta }\\
    &\forall k: \M_{d - d_{wi}}(g_{wi} y^k) \succeq 0 & & \forall i = 1, \ldots, N_c^w. \label{eq:peak_lmi_un_disc_d}
\end{align}
\end{subequations}

Constraint \eqref{eq:peak_lmi_un_time_limit} enforces the time limit constraint on occupation measures \eqref{eq:peak_meas_un_disc_time}. The structure of \eqref{eq:peak_lmi_un_disc} is similar to \eqref{eq:peak_lmi_un} with the affine, moment matrix and localizing matrix constraints.


\subsection{Discrete Example}

An example to demonstrate uncertain discrete peak estimation is to minimize $x_2$ on the following subsystems,
\begin{subequations}
\label{eq:poly_system}
\begin{align}
    f_1(x,w) &= \begin{bmatrix}-0.3x_1 + 0.8 x_2 + 0.1 x_1 x_2 \\ -0.75 x_1 - 0.3 x_2 +w \end{bmatrix}  \\
    f_2(x, w) &= \begin{bmatrix}0.8x_1 + 0.5 x_2 - 0.01 x_1^2 \\ -0.5 x_1 + 0.8 x_2 -0.01x_1 x_2 + w\end{bmatrix}.
\end{align}
\end{subequations}
The space under consideration is $X = [-3, 3]^2$, and the time varying uncertainty $w_t$ satisfies $w_t \in [-0.2, 0.2] = \Delta$. The valid regions for subsystems of \eqref{eq:poly_system} are $X_1 = X$ and $X_2 = X \cap (x_1 \geq 0)$. When $x_1 \geq 0$ the system may switch arbitrarily between dynamics $f_1$ and $f_2$, but when $x_2 < 0$, the system only follows dynamics $f_1$. Figure \ref{fig:poly} visualizes minimizing $x_2$ starting from the initial set $X_0 = \{x \mid (x_1 + 1.5)^2 + x_2^2 = 0.16\}$ between discrete times $t \in 0, \ldots, T$ with $T = 50$. A fourth order LMI relaxation of \eqref{eq:peak_meas_un_disc_obj} is solved aiming to maximize $p(x) = -x_2$. With $w=0$ in Fig. \ref{fig:poly_std} the bound is $P^* \leq 1.215$ ($\min x_2 \geq -1.215$), while the time varying $w$ in Fig. \ref{fig:poly_d} yields a bound of $P^* \leq 1.837$. 
\begin{figure}[ht]
     \centering
     \begin{subfigure}[b]{0.48\linewidth}
         \centering
         \includegraphics[width=\linewidth]{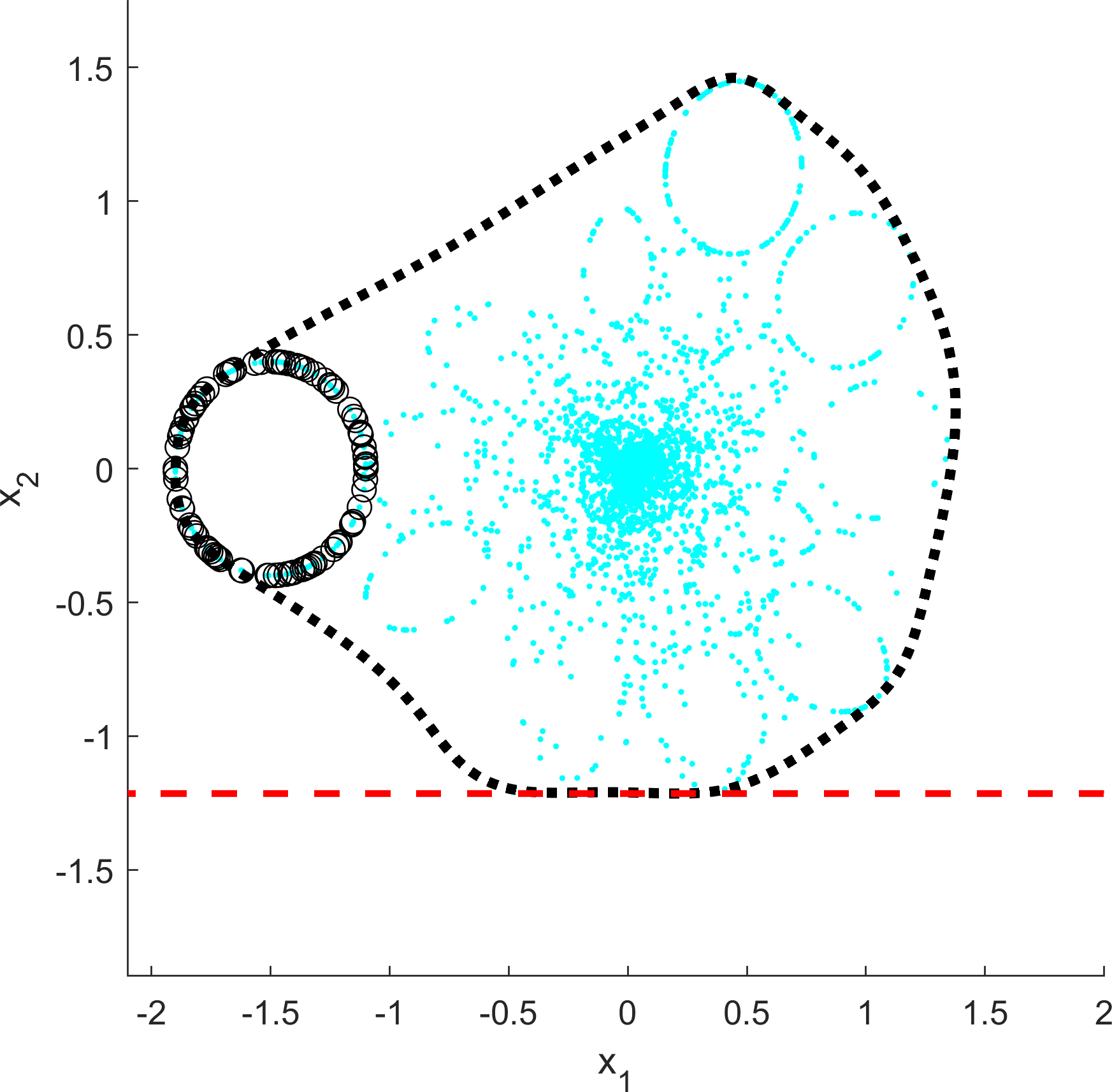}
         \caption{$w_t = 0$ }
         \label{fig:poly_std}
     \end{subfigure}
     \;
     \begin{subfigure}[b]{0.48\linewidth}
         \centering
         \includegraphics[width=\linewidth]{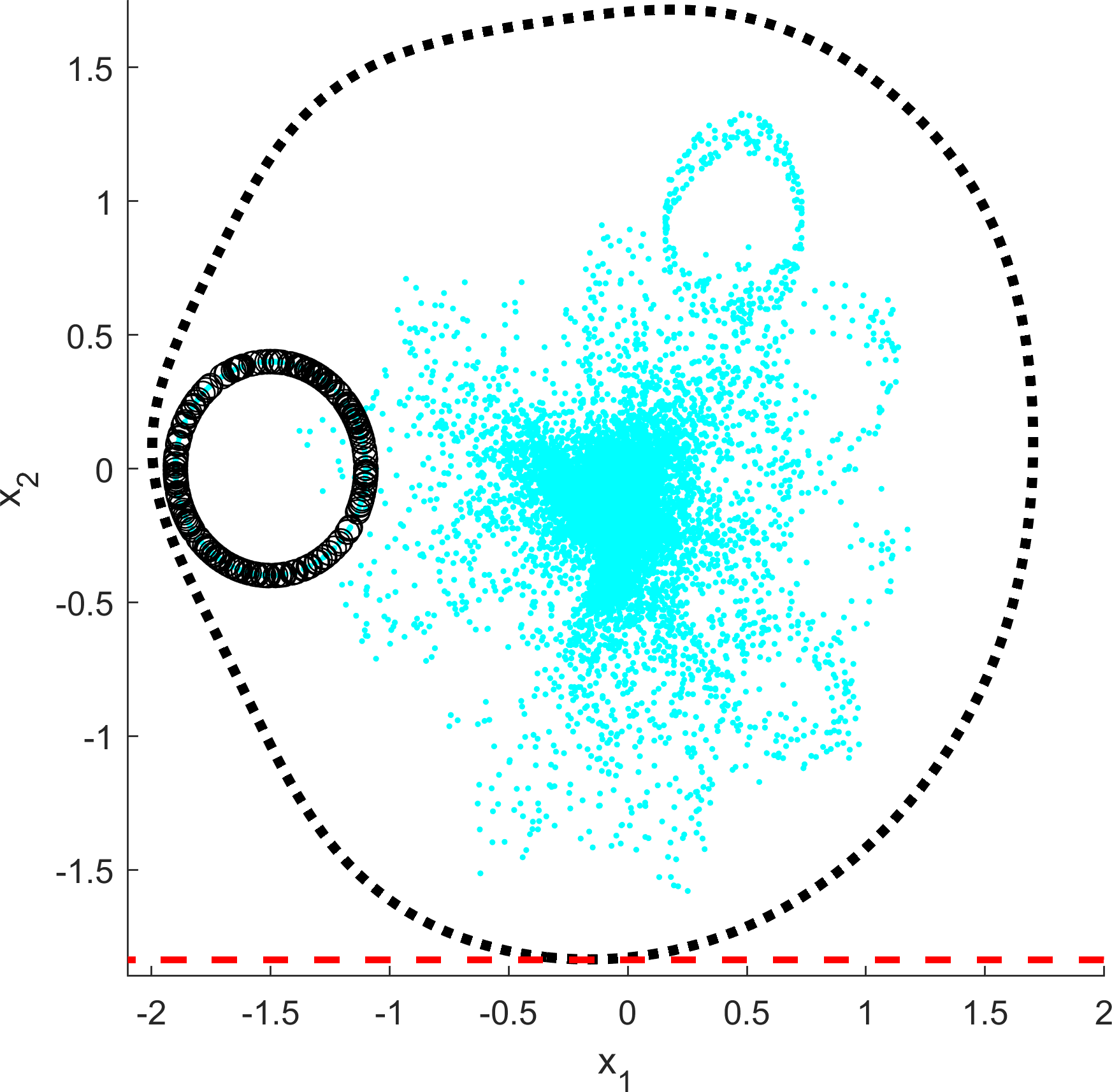}
         \caption{$w_t \in [-0.2, 0.2]$}
         \label{fig:poly_d}
     \end{subfigure}
      \caption{\label{fig:poly} Minimize $x_2$ on system  \eqref{eq:poly_system}}
\end{figure}

\section{Safety Analysis}
\label{sec:safety}

The work in \cite{miller2020recovery} introduced the concept of `safety margins' that are solvable through peak estimation to certify safety of trajectories. Assume that $X_u = \{x \mid p_i(x) \geq 0, \ i = 1, \ldots, N_u\}$ is a basic semialgebraic set with $N_u$ constraints defining an unsafe set. $X_u$ may be equivalently redefined as $X_u = \{x \mid \min_i p_i(x) \geq 0\}$. If the maximum value of $\min_i p_i(x)$ is negative for all points on trajectories starting from $x_0 \in X_0$ in times $t \in [0, T]$, then all trajectories are certifiably safe. The quantity of a `safety margin' is an upper bound for $\min_i p_i(x) \geq 0$ which may be found through LMI approximations. Finding the safety margin is an instance of maximin optimization, aiming to maximize the minimum of a set of functions.

A maximin optimization problem may be considered by replacing objectives \eqref{eq:peak_meas_un_obj} or \eqref{eq:peak_meas_un_disc_obj} with,
\begin{subequations}
\begin{align}
    & \textrm{max}_{q \in \R} \quad q  \label{eq:maximin_obj} \\
    & q \leq \inp{p_i}{\mu_p} & i = 1, \ldots, N_u.
    \end{align}
\end{subequations}

The dual formulation introduces variables $\beta \in \R_+^{N_u}$ as nonnegative multipliers. Constraints \eqref{eq:peak_cont_un_p} and \eqref{eq:peak_cont_disc_flow} are then replaced by, 
\begin{equation}
    v \geq \beta^T p(x) = \sum_{i=1}^{N_u} \beta_i p_i(x).
\end{equation}
over the valid region ($[0, T] \times X \times \Theta$ or $X \times \Theta$).

The value $q$ is a lower bound for all of the expectations $\inp{p_i}{\mu_p}$. A negative optimal value of $q$ for any degree of an LMI relaxation is sufficient to certify safety.

An example of a successful safety margin under uncertainty is depicted in Figure \ref{fig:flow_half}. The system under consideration is based on Example 1 of \cite{prajna2004safety}, with dynamics for time-varying $w$,
\begin{equation}
    \dot{x} = f(x, w) = \begin{bmatrix}x_2 \\ -x_1 + \frac{w}{3} x_1^3 - x_2 \end{bmatrix}.
\end{equation}

For trajectories originating in $X_0 = \{x \mid (x_1 - 1.5)^2 + x_2^2 \leq 0.4^2\}$, it is desired to determine if any trajectory reaches the half-circle unsafe set in red $X_u = \{x \mid p_1(x) = x_1^2 - (x_2 + 0.5)^2 \leq 0.25, \ p_2(x) = \frac{\sqrt{2}}{2} (x_1 + x_2 + 0.5) \geq 0$. When $w = 1$ is constant, the 5th-order LMI (relaxation of maximin peak estimation with infinite time) computes a safety margin of $p^*_5 = -0.1417 < 0$ certifying safety of all trajectories. This value is nearly optimal, and the trajectory starting at the blue circle in \ref{fig:flow_half_std} approximately maximizes $\min_i p_i(x)$ as recovered by Alg. 1 of \cite{miller2020recovery}. The black contour is the auxiliary function level set $\{x \mid v(x) = p^*_5\}$, and the red contour is the level set of $\{x \mid \min_i p_i(x) = p^*_5\}$.

The time-varying case where $w \in [0.5, 1.5]$ is shown in Figure \ref{fig:flow_half_b}. The safety margin of $p^*_5 = -0.0784 < 0 $ is computed at the 5th-order LMI relaxation.

\begin{figure}[ht]
     \centering
     \begin{subfigure}[b]{0.48\linewidth}
         \centering
         \includegraphics[width=\linewidth]{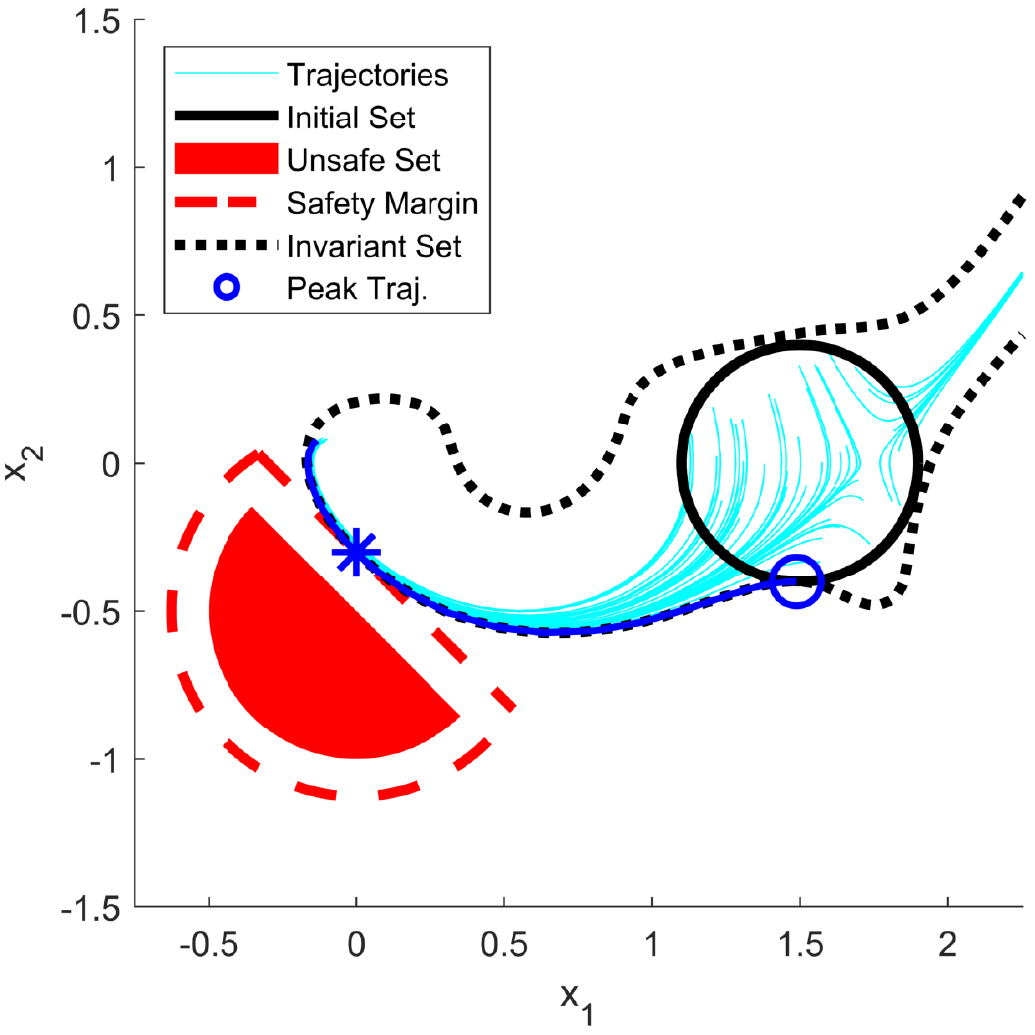}
         \caption{$w = 1, \ p_5^* = -0.1417$ }
         \label{fig:flow_half_std}
     \end{subfigure}
     \;
     \begin{subfigure}[b]{0.48\linewidth}
         \centering
         \includegraphics[width=\linewidth]{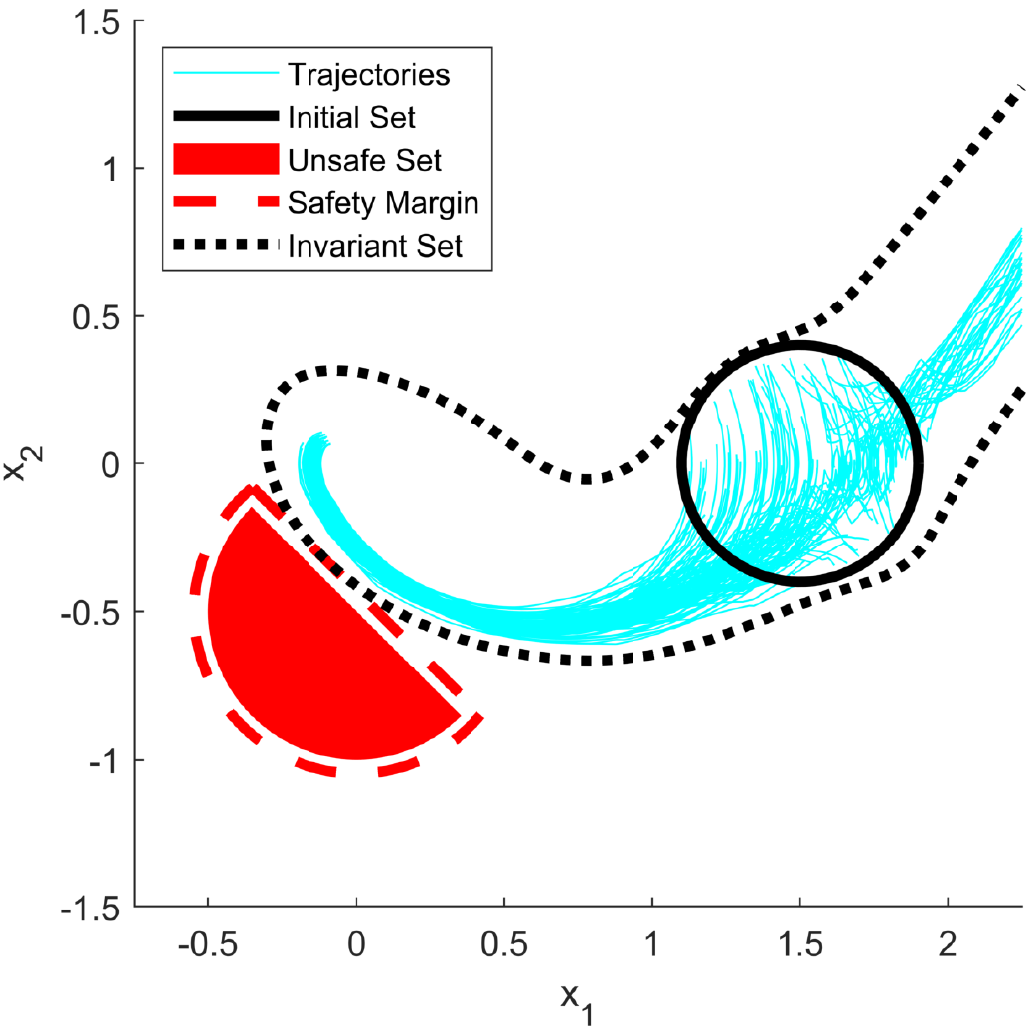}
         \caption{$w \in [-0.5, 1.5], \ p_6^* = -0.0487$}
         \label{fig:flow_half_b}
     \end{subfigure}
      \caption{\label{fig:flow_half} Safety margins on half-circle set}
\end{figure}


\section{Conclusion}

\label{sec:conclusion}


The problem of peak estimation with uncertainty may be bounded by the optimal value of an infinite-dimensional LP in occupation measures. This LP is then approximated by the moment-SOS hierarchy and Linear Matrix Inequalities. Time-independent and time-dependent  uncertainties  are incorporated into this measure framework for continuous and discrete systems. Future work includes uncertain peak estimation for safety verification and hybrid systems, and also exploiting specialized uncertainty structures.

\bibliographystyle{IEEEtran}
\bibliography{peak_uncertain_reference.bib}
\end{document}